\documentclass{amsart}
\usepackage{amsmath}
\usepackage{amsfonts}

\newtheorem{theorem}{Theorem}
\theoremstyle{plain}

\newtheorem{corollary}{Corollary}

\newtheorem{definition}{Definition}

\newtheorem{lemma}{Lemma}

\numberwithin{equation}{section}
\input{tcilatex}

\begin{document}
\title[On generalization of different type inequalities]{On generalization
of different type inequalities for harmonically quasi-convex functions via
fractional integrals}
\author{\.{I}mdat \.{I}\c{s}can}
\address{Department of Mathematics, Faculty of Sciences and Arts, Giresun
University, Giresun, Turkey}
\email{imdat.iscan@giresun.edu.tr}

\begin{abstract}
In this paper, we obtained some new estimates on generalization of Hadamard,
Ostrowski and Simpson-like type inequalities for harmonically quasi-convex
functions via Riemann Liouville fractional integral.
\end{abstract}

\maketitle

\section{Introduction}

In this section we will present definitions and some results used in this
paper.

A function $f:I\subseteq \mathbb{R\rightarrow R}$ is said to be a convex
function if 
\begin{equation*}
f(\lambda x+(1-\lambda )y)\leq \lambda f(x)+(1-\lambda )f(y)
\end{equation*}

holds for all $x,y\in I$ and $\lambda \in \lbrack 0,1].$

The notion of quasi-convex functions generalizes the notion of convex
functions. More precisely, a function $f:[u,v]\mathbb{\rightarrow R}$ is
said quasi-convex on $[u,v]$ if 
\begin{equation*}
f\left( \lambda x+(1-\lambda )y\right) \leq \sup \left\{ f(x),f(y)\right\} ,
\end{equation*}%
for any $x,y\in \lbrack u,v]$ and $\lambda \in \left[ 0,1\right] .$ Clearly,
any convex function is a quasi-convex function. Furthermore, there exist
quasi-convex functions which are not \ convex (see \cite{I07}).

Following inequalities are well known in the literature as Hermite-Hadamard
inequality, Ostrowski inequality and Simpson inequality respectively:

\begin{theorem}
Let $f:I\subseteq \mathbb{R\rightarrow R}$ be a convex function defined on
the interval $I$ of real numbers and $u,v\in I$ with $u<v$. The following
double inequality holds%
\begin{equation}
f\left( \frac{u+v}{2}\right) \leq \frac{1}{v-u}\dint\limits_{u}^{v}f(x)dx%
\leq \frac{f(u)+f(v)}{2}\text{.}  \label{1-1}
\end{equation}
\end{theorem}

\begin{theorem}
Let $f:I\subseteq \mathbb{R\rightarrow R}$ be a mapping differentiable in $%
I^{\circ },$ the interior of I, and let $u,v\in I^{\circ }$ with $u<v.$ If $%
\left\vert f^{\prime }(x)\right\vert \leq M$ for all $x\in \left[ u,v\right]
,$ then the following inequality holds%
\begin{equation*}
\left\vert f(x)-\frac{1}{v-u}\dint\limits_{u}^{v}f(t)dt\right\vert \leq 
\frac{M}{v-u}\left[ \frac{\left( x-u\right) ^{2}+\left( v-x\right) ^{2}}{2}%
\right]
\end{equation*}%
for all $x\in \left[ u,v\right] .$
\end{theorem}

\begin{theorem}
Let $f:\left[ u,v\right] \mathbb{\rightarrow R}$ be a four times
continuously differentiable mapping on $\left( u,v\right) $ and $\left\Vert
f^{(4)}\right\Vert _{\infty }=\underset{x\in \left( u,v\right) }{\sup }%
\left\vert f^{(4)}(x)\right\vert <\infty .$ Then the following inequality
holds:%
\begin{equation*}
\left\vert \frac{1}{3}\left[ \frac{f(u)+f(v)}{2}+2f\left( \frac{u+v}{2}%
\right) \right] -\frac{1}{v-u}\dint\limits_{u}^{v}f(x)dx\right\vert \leq 
\frac{1}{2880}\left\Vert f^{(4)}\right\Vert _{\infty }\left( v-u\right) ^{4}.
\end{equation*}
\end{theorem}

In \cite{I13}, Iscan defined harmonically convex functions and gave some
Hermite-Hadamard type inequalities for this class of functions

\begin{definition}
Let $I\subset 
%TCIMACRO{\U{211d} }%
%BeginExpansion
\mathbb{R}
%EndExpansion
\backslash \left\{ 0\right\} $ be a real interval. A function $%
f:I\rightarrow 
%TCIMACRO{\U{211d} }%
%BeginExpansion
\mathbb{R}
%EndExpansion
$ is said to be harmonically convex, if \ 
\begin{equation}
f\left( \frac{xy}{\lambda x+(1-\lambda )y}\right) \leq \lambda
f(y)+(1-\lambda )f(x)  \label{1-1a}
\end{equation}%
for all $x,y\in I$ and $\lambda \in \lbrack 0,1]$. If the inequality in (\ref%
{1-1a}) is reversed, then $f$ is said to be harmonically concave.
\end{definition}

In \cite{ZJF13}, Zhang et al. defined the harmonically quasi-convex function
and supplied several properties of this kind of functions.

\begin{definition}
A function $f:I\subseteq \left( 0,\infty \right) \rightarrow \left[ 0,\infty
\right) $ is said to be harmonically convex, if \ 
\begin{equation*}
f\left( \frac{xy}{\lambda x+(1-\lambda )y}\right) \leq \sup \left\{
f(x),f(y)\right\}
\end{equation*}%
for all $x,y\in I$ and $\lambda \in \lbrack 0,1]$.
\end{definition}

We would like to point out that any harmonically convex function on $%
I\subseteq \left( 0,\mathbb{\infty }\right) $ is a harmonically quasi-convex
function, but not conversely. For example, the function 
\begin{equation*}
f(x)=%
\begin{cases}
1, & x\in (0,1]; \\ 
(x-2)^{2}, & x\in \lbrack 1,4].%
\end{cases}%
\end{equation*}%
is harmonically quasi-convex on $(0,4]$, but it is not harmonically convex
on $(0,4]$.

Let us recall the following special functions:

(1) The Beta function:%
\begin{equation*}
\beta \left( x,y\right) =\frac{\Gamma (x)\Gamma (y)}{\Gamma (x+y)}%
=\dint\limits_{0}^{1}t^{x-1}\left( 1-t\right) ^{y-1}dt,\ \ x,y>0,
\end{equation*}

(2) The hypergeometric function:%
\begin{equation*}
_{2}F_{1}\left( a,b;c;z\right) =\frac{1}{\beta \left( b,c-b\right) }%
\dint\limits_{0}^{1}t^{b-1}\left( 1-t\right) ^{c-b-1}\left( 1-zt\right)
^{-a}dt,\ c>b>0,\ \left\vert z\right\vert <1\text{ (see \cite{KST06}).}
\end{equation*}

We give some necessary definitions and mathematical preliminaries of
fractional calculus theory which are used throughout this paper.

\begin{definition}
Let $f\in L\left[ u,v\right] $. The Riemann-Liouville integrals $%
J_{u+}^{\alpha }f$ and $J_{v-}^{\alpha }f$ of oder $\alpha >0$ with $u\geq 0$
are defined by

\begin{equation*}
J_{u+}^{\alpha }f(b)=\frac{1}{\Gamma (\alpha )}\dint\limits_{u}^{b}\left(
b-t\right) ^{\alpha -1}f(t)dt,\ b>u
\end{equation*}

and

\begin{equation*}
J_{v-}^{\alpha }f(a)=\frac{1}{\Gamma (\alpha )}\dint\limits_{a}^{v}\left(
t-a\right) ^{\alpha -1}f(t)dt,\ a<v
\end{equation*}%
respectively, where $\Gamma (\alpha )$ is the Gamma function defined by $%
\Gamma (\alpha )=$ $\dint\limits_{0}^{\infty }e^{-t}t^{\alpha -1}dt$ and $%
J_{u+}^{0}f(x)=J_{v-}^{0}f(x)=f(x).$
\end{definition}

In the case of $\alpha =1$, the fractional integral reduces to the classical
integral. In recent years, many athors have studied errors estimations for
Hermite-Hadamard, Ostrowski and Simpson inequalities; for refinements,
counterparts, generalization see \cite%
{ADDC10,ADK10,AH11,I13b,I13c,I13d,I13e,SA11,S12,SO12,SOS12,SSYB11}.

In \cite{I13a}, Iscan gave the following new identity for differentiable
functions and obtained some new general integral inequalities for
harmonically convex functions via fractional integrals.

\begin{lemma}
\label{2.1}Let $f:I\subseteq \left( 0,\infty \right) \rightarrow 
%TCIMACRO{\U{211d} }%
%BeginExpansion
\mathbb{R}
%EndExpansion
$ be a differentiable function on $I^{\circ }$ and $a,b\in I$ with $a<b$. If 
$f^{\prime }\in L[a,b]$ then for all $x\in \lbrack a,b]$ , $\lambda \in %
\left[ 0,1\right] $ and $\alpha >0$ we have:%
\begin{eqnarray*}
&&I_{f,g}\left( x,\lambda ,\alpha ,a,b\right) =\frac{\left( x-a\right)
^{\alpha +1}}{(ax)^{\alpha -1}}\dint\limits_{0}^{1}\frac{t^{\alpha }-\lambda 
}{A_{t}^{2}(a,x)}f^{\prime }\left( \frac{ax}{A_{t}(a,x)}\right) dt \\
&&-\frac{\left( b-x\right) ^{\alpha +1}}{(bx)^{\alpha -1}}%
\dint\limits_{0}^{1}\frac{t^{\alpha }-\lambda }{A_{t}^{2}(b,x)}f^{\prime
}\left( \frac{bx}{A_{t}(b,x)}\right) dt,
\end{eqnarray*}%
where $A_{t}(u,x)=tu+(1-t)x$ and 
\begin{eqnarray}
&&I_{f,g}\left( x,\lambda ,\alpha ,a,b\right)  \label{1-2} \\
&=&\left( 1-\lambda \right) \left[ \left( \frac{x-a}{ax}\right) ^{\alpha
}+\left( \frac{b-x}{bx}\right) ^{\alpha }\right] f(x)+\lambda \left[ \left( 
\frac{x-a}{ax}\right) ^{\alpha }f(a)+\left( \frac{b-x}{bx}\right) ^{\alpha
}f(b)\right]  \notag \\
&&-\Gamma \left( \alpha +1\right) \left[ J_{1/x+}^{\alpha }\left( f\circ
g\right) (1/a)+J_{1/x-}^{\alpha }\left( f\circ g\right) (1/b)\right] \text{, 
}g(u)=1/u.  \notag
\end{eqnarray}
\end{lemma}

In this paper, by using of Lemma \ref{2.1}, we obtained a generalization of
Hadamard, Ostrowski and Simpson type inequalities for harmonically
quasi-convex functions via Riemann Liouville fractional integral.

\section{Main Results}

Let $f:I\subseteq \left( 0,\infty \right) \rightarrow 
%TCIMACRO{\U{211d} }%
%BeginExpansion
\mathbb{R}
%EndExpansion
$ be a differentiable function on $I^{\circ }$, the interior of $I$,
throughout this section we will take the notation $I_{f,g}\left( x,\lambda
,\alpha ,a,b\right) $ as in (\ref{1-2}), where $a,b\in I$ with $a<b$, $\
x\in \lbrack a,b]$ , $\lambda \in \left[ 0,1\right] $, $g(u)=1/u$, $\alpha
>0 $ and $\Gamma $ is Euler Gamma function. In order to prove our main
results we need the following identity.

\begin{theorem}
\label{2.2}Let $f:$ $I\subseteq \left( 0,\infty \right) \rightarrow 
%TCIMACRO{\U{211d} }%
%BeginExpansion
\mathbb{R}
%EndExpansion
$ be a differentiable function on $I^{\circ }$ such that $f^{\prime }\in
L[a,b]$, where $a,b\in I^{\circ }$ with $a<b$. If $|f^{\prime }|^{q}$ is
harmonically quasi-convex on $[a,b]$ for some fixed $q\geq 1$, then for $%
x\in \lbrack a,b]$, $\lambda \in \left[ 0,1\right] $ and $\alpha >0$ the
following inequality holds%
\begin{eqnarray}
&&\left\vert I_{f,g}\left( x,\lambda ,\alpha ,a,b\right) \right\vert \leq
C_{1}^{1-1/q}\left( \alpha ,\lambda \right)  \notag \\
&&\times \left\{ \frac{\left( x-a\right) ^{\alpha +1}}{(ax)^{\alpha -1}x^{2q}%
}\left[ \sup \left\{ \left\vert f^{\prime }\left( x\right) \right\vert
^{q},\left\vert f^{\prime }\left( a\right) \right\vert ^{q}\right\} \right]
^{1/q}C_{2}^{1/q}\left( \alpha ,\lambda ,q,\frac{a}{x}\right) \right.
\label{2-2} \\
&&\left. +\frac{\left( b-x\right) ^{\alpha +1}}{(bx)^{\alpha -1}b^{2q}}\left[
\sup \left\{ \left\vert f^{\prime }\left( x\right) \right\vert
^{q},\left\vert f^{\prime }\left( a\right) \right\vert ^{q}\right\} \right]
^{1/q}C_{3}^{1/q}\left( \alpha ,\lambda ,q,\frac{x}{b}\right) \right\} , 
\notag
\end{eqnarray}%
where 
\begin{equation*}
C_{1}\left( \alpha ,\lambda \right) =\frac{2\alpha \lambda ^{1+\frac{1}{%
\alpha }}+1}{\alpha +1}-\lambda ,
\end{equation*}%
\begin{eqnarray*}
&&C_{2}\left( \alpha ,\lambda ,q,\frac{a}{x}\right) \\
&=&\frac{1}{\alpha +1}._{2}F_{1}(2q;\alpha +1;\alpha +2,1-\frac{a}{x}%
)-\lambda ._{2}F_{1}(2q;1;2,1-\frac{a}{x}) \\
&&+2\lambda ^{1+\frac{1}{\alpha }}.\left[ _{2}F_{1}(2q;1;2,\lambda
^{1/\alpha }\left( 1-\frac{a}{x}\right) )-\frac{1}{\alpha +1}%
._{2}F_{1}(2q;\alpha +1;\alpha +2,\lambda ^{1/\alpha }\left( 1-\frac{a}{x}%
\right) )\right] ,
\end{eqnarray*}%
\begin{eqnarray*}
&&C_{3}\left( \alpha ,\lambda ,q,\frac{x}{b}\right) \\
&=&\frac{1}{\alpha +1}._{2}F_{1}(2q;1;\alpha +2,1-\frac{x}{b})-\lambda
._{2}F_{1}(2q;1;2,1-\frac{x}{b}) \\
&&+2\lambda ^{1+\frac{1}{\alpha }}.\left[ _{2}F_{1}(2q;1;2,\lambda
^{1/\alpha }\left( 1-\frac{x}{b}\right) )-\frac{1}{\alpha +1}%
._{2}F_{1}(2q;1;\alpha +2,\lambda ^{1/\alpha }\left( 1-\frac{x}{b}\right) )%
\right] .
\end{eqnarray*}
\end{theorem}

\begin{proof}
Since$\left\vert f^{\prime }\right\vert ^{q}$ is harmonically quasi-convex
on $[a,b]$, from Lemma \ref{2.1}, property of the modulus and using the
power-mean inequality we have%
\begin{eqnarray*}
&&\left\vert I_{f,g}\left( x,\lambda ,\alpha ,a,b\right) \right\vert \leq 
\frac{\left( x-a\right) ^{\alpha +1}}{(ax)^{\alpha -1}}\dint\limits_{0}^{1}%
\frac{\left\vert t^{\alpha }-\lambda \right\vert }{A_{t}^{2}(a,x)}\left\vert
f^{\prime }\left( \frac{ax}{A_{t}(a,x)}\right) \right\vert dt \\
&&+\frac{\left( b-x\right) ^{\alpha +1}}{(bx)^{\alpha -1}}%
\dint\limits_{0}^{1}\frac{\left\vert t^{\alpha }-\lambda \right\vert }{%
A_{t}^{2}(b,x)}\left\vert f^{\prime }\left( \frac{bx}{A_{t}(b,x)}\right)
\right\vert dt \\
&\leq &\frac{\left( x-a\right) ^{\alpha +1}}{(ax)^{\alpha -1}}\left(
\dint\limits_{0}^{1}\left\vert t^{\alpha }-\lambda \right\vert dt\right) ^{1-%
\frac{1}{q}}\left( \dint\limits_{0}^{1}\frac{\left\vert t^{\alpha }-\lambda
\right\vert }{A_{t}^{2q}(a,x)}\left\vert f^{\prime }\left( \frac{ax}{%
A_{t}(a,x)}\right) \right\vert ^{q}dt\right) ^{\frac{1}{q}} \\
&&+\frac{\left( b-x\right) ^{\alpha +1}}{(bx)^{\alpha -1}}\left(
\dint\limits_{0}^{1}\left\vert t^{\alpha }-\lambda \right\vert dt\right) ^{1-%
\frac{1}{q}}\left( \dint\limits_{0}^{1}\frac{\left\vert t^{\alpha }-\lambda
\right\vert }{A_{t}^{2q}(b,x)}\left\vert f^{\prime }\left( \frac{bx}{%
A_{t}(b,x)}\right) \right\vert ^{q}dt\right) ^{\frac{1}{q}}
\end{eqnarray*}%
\begin{eqnarray}
&\leq &C_{1}^{1-1/q}\left( \alpha ,\lambda \right) \left\{ \frac{\left(
x-a\right) ^{\alpha +1}}{(ax)^{\alpha -1}}\left[ \sup \left\{ \left\vert
f^{\prime }\left( x\right) \right\vert ^{q},\left\vert f^{\prime }\left(
a\right) \right\vert ^{q}\right\} \right] ^{1/q}\left( \dint\limits_{0}^{1}%
\frac{\left\vert t^{\alpha }-\lambda \right\vert }{A_{t}^{2q}(a,x)}dt\right)
^{1/q}\right.  \notag \\
&&\left. +\frac{\left( b-x\right) ^{\alpha +1}}{(bx)^{\alpha -1}}\left[ \sup
\left\{ \left\vert f^{\prime }\left( x\right) \right\vert ^{q},\left\vert
f^{\prime }\left( a\right) \right\vert ^{q}\right\} \right] ^{1/q}\left(
\dint\limits_{0}^{1}\frac{\left\vert t^{\alpha }-\lambda \right\vert }{%
A_{t}^{2q}(a,x)}dt\right) ^{1/q}\right\}  \label{2-2a}
\end{eqnarray}%
where by simple computation we obtain%
\begin{eqnarray}
C_{1}\left( \alpha ,\lambda \right) &=&\dint\limits_{0}^{1}\left\vert
t^{\alpha }-\lambda \right\vert dt  \notag \\
&=&\frac{2\alpha \lambda ^{1+\frac{1}{\alpha }}+1}{\alpha +1}-\lambda ,
\label{2-2b}
\end{eqnarray}%
\begin{equation}
\dint\limits_{0}^{1}\frac{\left\vert t^{\alpha }-\lambda \right\vert }{%
A_{t}^{2q}(a,x)}dt  \label{2-2c}
\end{equation}%
\begin{eqnarray}
&=&x^{-2q}\left\{ \frac{1}{\alpha +1}._{2}F_{1}(2q;\alpha +1;\alpha +2,1-%
\frac{a}{x})-\lambda ._{2}F_{1}(2q;1;2,1-\frac{a}{x})\right.  \notag \\
&&+2\lambda ^{1+\frac{1}{\alpha }}.\left[ _{2}F_{1}(2q;1;2,\lambda
^{1/\alpha }\left( 1-\frac{a}{x}\right) )-\frac{1}{\alpha +1}%
._{2}F_{1}(2q;\alpha +1;\alpha +2,\lambda ^{1/\alpha }\left( 1-\frac{a}{x}%
\right) )\right] ,  \notag
\end{eqnarray}%
and%
\begin{equation}
\dint\limits_{0}^{1}\frac{\left\vert t^{\alpha }-\lambda \right\vert }{%
A_{t}^{2q}(b,x)}dt  \label{2-2d}
\end{equation}%
\begin{eqnarray}
&=&b^{-2q}\left\{ \frac{1}{\alpha +1}._{2}F_{1}(2q;1;\alpha +2,1-\frac{x}{b}%
)-\lambda ._{2}F_{1}(2q;1;2,1-\frac{x}{b})\right.  \notag \\
&&+2\lambda ^{1+\frac{1}{\alpha }}.\left[ _{2}F_{1}(2q;1;2,\lambda
^{1/\alpha }\left( 1-\frac{x}{b}\right) )-\frac{1}{\alpha +1}%
._{2}F_{1}(2q;1;\alpha +2,\lambda ^{1/\alpha }\left( 1-\frac{x}{b}\right) )%
\right] ,  \notag
\end{eqnarray}%
Hence, If we use (\ref{2-2b}), (\ref{2-2c}) and (\ref{2-2d}) in (\ref{2-2a}%
), we obtain the desired result. This completes the proof.
\end{proof}

\begin{corollary}
Under the assumptions of Theorem \ref{2.2} with $q=1,$ the inequality (\ref%
{2-2}) reduced to the following inequality%
\begin{eqnarray*}
&&\left\vert I_{f,g}\left( x,\lambda ,\alpha ,a,b\right) \right\vert \\
&\leq &\left\{ \frac{\left( x-a\right) ^{\alpha +1}}{(ax)^{\alpha -1}x^{2}}%
\left[ \sup \left\{ \left\vert f^{\prime }\left( x\right) \right\vert
,\left\vert f^{\prime }\left( a\right) \right\vert \right\} \right]
C_{2}\left( \alpha ,\lambda ,1,\frac{a}{x}\right) \right. \\
&&\left. +\frac{\left( b-x\right) ^{\alpha +1}}{(bx)^{\alpha -1}b^{2}}\left[
\sup \left\{ \left\vert f^{\prime }\left( x\right) \right\vert ,\left\vert
f^{\prime }\left( a\right) \right\vert \right\} \right] ^{1/q}C_{3}\left(
\alpha ,\lambda ,1,\frac{x}{b}\right) \right\} ,
\end{eqnarray*}
\end{corollary}

\begin{corollary}
Under the assumptions of Theorem \ref{2.2} with $\alpha =1,$ the inequality (%
\ref{2-2}) reduced to the following inequality%
\begin{eqnarray*}
&&\left( \frac{ab}{b-a}\right) \left\vert I_{f,g}\left( x,\lambda ,\alpha
,a,b\right) \right\vert \\
&=&\left\vert \left( 1-\lambda \right) f(x)+\lambda \left[ \frac{%
b(x-a)f(a)+a(b-x)f(b)}{x(b-a)}\right] -\frac{ab}{b-a}\dint\limits_{a}^{b}%
\frac{f\left( u\right) }{u^{2}}du\right\vert \\
&\leq &\left( \frac{ab}{b-a}\right) \left( \frac{2\lambda ^{2}-2\lambda +1}{2%
}\right) ^{1-\frac{1}{q}}\left\{ \frac{\left( x-a\right) ^{2}}{x^{2q}}\left[
\sup \left\{ \left\vert f^{\prime }\left( x\right) \right\vert
^{q},\left\vert f^{\prime }\left( a\right) \right\vert ^{q}\right\} \right]
^{1/q}C_{2}^{1/q}\left( 1,\lambda ,q,\frac{a}{x}\right) \right. \\
&&\left. +\frac{\left( b-x\right) ^{2}}{b^{2q}}\left[ \sup \left\{
\left\vert f^{\prime }\left( x\right) \right\vert ^{q},\left\vert f^{\prime
}\left( a\right) \right\vert ^{q}\right\} \right] ^{1/q}C_{3}^{1/q}\left(
1,\lambda ,q,\frac{x}{b}\right) \right\} ,
\end{eqnarray*}%
specially for $x=H=2ab/(a+b)$, we get%
\begin{eqnarray*}
&&\left\vert \left( 1-\lambda \right) f\left( H\right) +\lambda \left( \frac{%
f(a)+f(b)}{2}\right) -\frac{ab}{b-a}\dint\limits_{a}^{b}\frac{f\left(
u\right) }{u^{2}}du\right\vert \leq \frac{b-a}{4ab}\left( \frac{2\lambda
^{2}-2\lambda +1}{2}\right) ^{1-\frac{1}{q}} \\
&&\times \left\{ \frac{a^{2}H^{2}}{H^{2q}}C_{2}^{1/q}\left( 1,\lambda ,q,%
\frac{a+b}{2b}\right) \left( \sup \left\{ \left\vert f^{\prime }\left(
H\right) \right\vert ^{q},\left\vert f^{\prime }\left( a\right) \right\vert
^{q}\right\} \right) ^{\frac{1}{q}}\right. \\
&&+\left. \frac{b^{2}H^{2}}{b^{2q}}C_{3}^{1/q}\left( 1,\lambda ,q,\frac{2a}{%
a+b}\right) \left( \sup \left\{ \left\vert f^{\prime }\left( H\right)
\right\vert ^{q},\left\vert f^{\prime }\left( b\right) \right\vert
^{q}\right\} \right) ^{\frac{1}{q}}\right\} .
\end{eqnarray*}
\end{corollary}

\begin{corollary}
In Theorem \ref{2.2},

(1) If we take $x=H=2ab/(a+b),\ \lambda =\frac{1}{3}$, then we get the
following Simpson type inequality for fractional integrals 
\begin{eqnarray*}
&&\left\vert \frac{1}{6}\left[ f(a)+4f\left( H\right) +f(b)\right] -\left( 
\frac{ab}{b-a}\right) ^{\alpha }2^{\alpha -1}\Gamma \left( \alpha +1\right) %
\left[ J_{1/H+}^{\alpha }\left( f\circ g\right) (1/a)+J_{1/H-}^{\alpha
}\left( f\circ g\right) (1/b)\right] \right\vert \\
&\leq &\frac{b-a}{4ab}C_{1}^{1-1/q}\left( \alpha ,\frac{1}{3}\right) \left\{ 
\frac{a^{2}H^{2}}{H^{2q}}\left[ \sup \left\{ \left\vert f^{\prime }\left(
H\right) \right\vert ^{q},\left\vert f^{\prime }\left( a\right) \right\vert
^{q}\right\} \right] ^{1/q}C_{2}^{1/q}\left( \alpha ,\frac{1}{3},q,\frac{a}{H%
}\right) \right. \\
&&\left. +\frac{b^{2}H^{2}}{b^{2q}}\left[ \sup \left\{ \left\vert f^{\prime
}\left( H\right) \right\vert ^{q},\left\vert f^{\prime }\left( a\right)
\right\vert ^{q}\right\} \right] ^{1/q}C_{3}^{1/q}\left( \alpha ,\frac{1}{3}%
,q,\frac{H}{b}\right) \right\} ,
\end{eqnarray*}%
specially for $\alpha =1$, we get%
\begin{eqnarray*}
&&\left\vert \frac{1}{6}\left[ f(a)+4f\left( H\right) +f(b)\right] -\frac{ab%
}{b-a}\dint\limits_{a}^{b}\frac{f(u)}{u^{2}}du\right\vert \leq \frac{b-a}{4ab%
}\left( \frac{5}{18}\right) ^{1-\frac{1}{q}} \\
&&\times \left\{ \frac{a^{2}H^{2}}{H^{2q}}C_{2}^{1/q}\left( 1,\frac{1}{3},q,%
\frac{a+b}{2b}\right) \left[ \sup \left\{ \left\vert f^{\prime }\left(
H\right) \right\vert ,\left\vert f^{\prime }\left( a\right) \right\vert
\right\} \right] ^{\frac{1}{q}}\right. \\
&&\left. +\frac{b^{2}H^{2}}{b^{2q}}C_{3}^{1/q}\left( 1,\frac{1}{3},q,\frac{2a%
}{a+b}\right) \left[ \sup \left\{ \left\vert f^{\prime }\left( H\right)
\right\vert ^{q},\left\vert f^{\prime }\left( b\right) \right\vert
^{q}\right\} \right] ^{\frac{1}{q}}\right\} .
\end{eqnarray*}%
(2) If we take $x=H=2ab/(a+b),\ \lambda =0,$ then we get the following
midpoint type inequality for fractional integrals 
\begin{eqnarray*}
&&\left\vert f\left( H\right) -\left( \frac{ab}{b-a}\right) ^{\alpha
}2^{\alpha -1}\Gamma \left( \alpha +1\right) \left[ J_{1/H+}^{\alpha }\left(
f\circ g\right) (1/a)+J_{1/H-}^{\alpha }\left( f\circ g\right) (1/b)\right]
\right\vert \\
&\leq &\frac{b-a}{4ab}\left( \frac{1}{\alpha +1}\right) ^{1-\frac{1}{q}%
}\left\{ \frac{a^{2}H^{2}}{H^{2q}}C_{2}^{1/q}\left( \alpha ,0,q,\frac{a+b}{2b%
}\right) \left[ \sup \left\{ \left\vert f^{\prime }\left( H\right)
\right\vert ^{q},\left\vert f^{\prime }\left( a\right) \right\vert
^{q}\right\} \right] ^{\frac{1}{q}}\right. \\
&&\left. +\frac{b^{2}H^{2}}{b^{2q}}C_{3}^{1/q}\left( \alpha ,0,q,\frac{2a}{%
a+b}\right) \left[ \sup \left\{ \left\vert f^{\prime }\left( H\right)
\right\vert ^{q},\left\vert f^{\prime }\left( b\right) \right\vert
^{q}\right\} \right] ^{\frac{1}{q}}\right\} ,
\end{eqnarray*}%
specially for $\alpha =1$, we get%
\begin{eqnarray*}
&&\left\vert f\left( H\right) -\frac{ab}{b-a}\dint\limits_{a}^{b}\frac{f(u)}{%
u^{2}}du\right\vert \\
&\leq &\frac{b-a}{4ab}\left( \frac{1}{2}\right) ^{1-\frac{1}{q}}\left\{ 
\frac{a^{2}H^{2}}{H^{2q}}C_{2}^{1/q}\left( 1,0,q,\frac{a+b}{2b}\right) \left[
\sup \left\{ \left\vert f^{\prime }\left( H\right) \right\vert
^{q},\left\vert f^{\prime }\left( a\right) \right\vert ^{q}\right\} \right]
^{\frac{1}{q}}\right. \\
&&\left. +\frac{b^{2}H^{2}}{b^{2q}}C_{3}^{1/q}\left( 1,0,q,\frac{2a}{a+b}%
\right) \left[ \sup \left\{ \left\vert f^{\prime }\left( H\right)
\right\vert ^{q},\left\vert f^{\prime }\left( b\right) \right\vert
^{q}\right\} \right] ^{\frac{1}{q}}\right\} .
\end{eqnarray*}%
(3) If we take $x=H=2ab/(a+b)$, $\lambda =1,$ then we get the following
trapezoid type inequality for fractional integrals 
\begin{eqnarray*}
&&\left\vert \frac{f(a)+f(b)}{2}-\left( \frac{ab}{b-a}\right) ^{\alpha
}2^{\alpha -1}\Gamma \left( \alpha +1\right) \left[ J_{1/H+}^{\alpha }\left(
f\circ g\right) (1/a)+J_{1/H-}^{\alpha }\left( f\circ g\right) (1/b)\right]
\right\vert \\
&\leq &\frac{b-a}{4ab}\left( \frac{\alpha }{\alpha +1}\right) ^{1-\frac{1}{q}%
}\left\{ \frac{a^{2}H^{2}}{H^{2q}}C_{2}^{1/q}\left( \alpha ,1,q,\frac{a+b}{2b%
}\right) \left[ \sup \left\{ \left\vert f^{\prime }\left( H\right)
\right\vert ^{q},\left\vert f^{\prime }\left( a\right) \right\vert
^{q}\right\} \right] ^{\frac{1}{q}}\right. \\
&&\left. +\frac{b^{2}H^{2}}{b^{2q}}C_{3}^{1/q}\left( \alpha ,1,q,\frac{2a}{%
a+b}\right) \left[ \sup \left\{ \left\vert f^{\prime }\left( H\right)
\right\vert ^{q},\left\vert f^{\prime }\left( b\right) \right\vert
^{q}\right\} \right] ^{\frac{1}{q}}\right\} ,
\end{eqnarray*}%
specially for $\alpha =1$, we get%
\begin{eqnarray*}
&&\left\vert \frac{f(a)+f(b)}{2}-\frac{ab}{b-a}\dint\limits_{a}^{b}\frac{f(u)%
}{u^{2}}du\right\vert \\
&\leq &\frac{b-a}{4ab}\left( \frac{1}{2}\right) ^{1-\frac{1}{q}}\left\{ 
\frac{a^{2}H^{2}}{H^{2q}}C_{2}^{1/q}\left( 1,1,q,\frac{a+b}{2b}\right) \left[
\sup \left\{ \left\vert f^{\prime }\left( H\right) \right\vert
^{q},\left\vert f^{\prime }\left( a\right) \right\vert ^{q}\right\} \right]
^{\frac{1}{q}}\right. \\
&&\left. +\frac{b^{2}H^{2}}{b^{2q}}C_{3}^{1/q}\left( 1,1,q,\frac{2a}{a+b}%
\right) \left[ \sup \left\{ \left\vert f^{\prime }\left( H\right)
\right\vert ^{q},\left\vert f^{\prime }\left( b\right) \right\vert
^{q}\right\} \right] ^{\frac{1}{q}}\right\} .
\end{eqnarray*}
\end{corollary}

\begin{corollary}
Let the assumptions of Theorem \ref{2.2} hold. If $\ \left\vert f^{\prime
}(x)\right\vert \leq M$ for all $x\in \left[ a,b\right] $ and $\lambda =0,$
then we get the following Ostrowski type inequality for fractional from the
inequality (\ref{2-2}) integrals%
\begin{eqnarray*}
&&\left\vert \left[ \left( \frac{x-a}{ax}\right) ^{\alpha }+\left( \frac{b-x%
}{bx}\right) ^{\alpha }\right] f(x)-\left( \frac{ab}{b-a}\right) ^{\alpha
}2^{\alpha -1}\Gamma \left( \alpha +1\right) \left[ J_{1/x+}^{\alpha }\left(
f\circ g\right) (1/a)+J_{1/x-}^{\alpha }\left( f\circ g\right) (1/b)\right]
\right\vert \\
&\leq &\frac{M}{\left( \alpha +1\right) ^{1-\frac{1}{q}}}\left[ \frac{\left(
x-a\right) ^{\alpha +1}}{(ax)^{\alpha -1}x^{2q}}C_{2}^{1/q}\left( \alpha
,0,q,\frac{a}{x}\right) +\frac{\left( b-x\right) ^{\alpha +1}}{(bx)^{\alpha
-1}b^{2q}}C_{3}^{1/q}\left( \alpha ,0,q,\frac{x}{b}\right) \right] .
\end{eqnarray*}
\end{corollary}

\begin{theorem}
\label{2.3}Let $f:$ $I\subseteq \left( 0,\infty \right) \rightarrow 
%TCIMACRO{\U{211d} }%
%BeginExpansion
\mathbb{R}
%EndExpansion
$ be a differentiable function on $I^{\circ }$ such that $f^{\prime }\in
L[a,b]$, where $a,b\in I^{\circ }$ with $a<b$. If $|f^{\prime }|^{q}$ is
harmonically quasi-convex on $[a,b]$ for some fixed $q\geq 1$, then for $%
x\in \lbrack a,b]$, $\lambda \in \left[ 0,1\right] $ and $\alpha >0$ the
following inequality holds%
\begin{eqnarray}
\left\vert I_{f,g}\left( x,\lambda ,\alpha ,a,b\right) \right\vert &\leq
&C_{2}\left( \alpha ,\lambda ,1,\frac{a}{x}\right) \frac{\left( x-a\right)
^{\alpha +1}}{(ax)^{\alpha -1}x^{2q}}\left[ \sup \left\{ \left\vert
f^{\prime }\left( x\right) \right\vert ^{q},\left\vert f^{\prime }\left(
a\right) \right\vert ^{q}\right\} \right] ^{1/q}  \label{2-3} \\
&&+C_{3}\left( \alpha ,\lambda ,1,\frac{x}{b}\right) \frac{\left( b-x\right)
^{\alpha +1}}{(bx)^{\alpha -1}b^{2q}}\left[ \sup \left\{ \left\vert
f^{\prime }\left( x\right) \right\vert ^{q},\left\vert f^{\prime }\left(
b\right) \right\vert ^{q}\right\} \right] ^{1/q}  \notag
\end{eqnarray}%
where $C_{2}$ and $C_{3}$ are defined as in Theorem \ref{2.2}.
\end{theorem}

\begin{proof}
Since$\left\vert f^{\prime }\right\vert ^{q}$ is harmonically quasi-convex
on $[a,b]$, from Lemma \ref{2.1},using property of the modulus and the
power-mean inequality we have%
\begin{eqnarray*}
&&\left\vert I_{f,g}\left( x,\lambda ,\alpha ,a,b\right) \right\vert \\
&\leq &\frac{\left( x-a\right) ^{\alpha +1}}{(ax)^{\alpha -1}}\left(
\dint\limits_{0}^{1}\frac{\left\vert t^{\alpha }-\lambda \right\vert }{%
A_{t}^{2}(a,x)}dt\right) ^{1-\frac{1}{q}}\left( \dint\limits_{0}^{1}\frac{%
\left\vert t^{\alpha }-\lambda \right\vert }{A_{t}^{2}(a,x)}\left\vert
f^{\prime }\left( \frac{ax}{A_{t}(a,x)}\right) \right\vert ^{q}dt\right) ^{%
\frac{1}{q}} \\
&&+\frac{\left( b-x\right) ^{\alpha +1}}{(bx)^{\alpha -1}}\left(
\dint\limits_{0}^{1}\frac{\left\vert t^{\alpha }-\lambda \right\vert }{%
A_{t}^{2}(b,x)}dt\right) ^{1-\frac{1}{q}}\left( \dint\limits_{0}^{1}\frac{%
\left\vert t^{\alpha }-\lambda \right\vert }{A_{t}^{2}(b,x)}\left\vert
f^{\prime }\left( \frac{bx}{A_{t}(b,x)}\right) \right\vert ^{q}dt\right) ^{%
\frac{1}{q}}
\end{eqnarray*}%
\begin{eqnarray*}
&\leq &C_{2}\left( \alpha ,\lambda ,1,\frac{a}{x}\right) \frac{\left(
x-a\right) ^{\alpha +1}}{(ax)^{\alpha -1}x^{2q}}\left[ \sup \left\{
\left\vert f^{\prime }\left( x\right) \right\vert ^{q},\left\vert f^{\prime
}\left( a\right) \right\vert ^{q}\right\} \right] ^{1/q} \\
&&+C_{3}\left( \alpha ,\lambda ,1,\frac{x}{b}\right) \frac{\left( b-x\right)
^{\alpha +1}}{(bx)^{\alpha -1}b^{2q}}\left[ \sup \left\{ \left\vert
f^{\prime }\left( x\right) \right\vert ^{q},\left\vert f^{\prime }\left(
b\right) \right\vert ^{q}\right\} \right] ^{1/q}
\end{eqnarray*}%
which completes the proof.
\end{proof}

\begin{corollary}
Under the assumptions of Theorem \ref{2.3} with $q=1,$ the inequality (\ref%
{2-3}) reduced to the following inequality%
\begin{eqnarray*}
\left\vert I_{f,g}\left( x,\lambda ,\alpha ,a,b\right) \right\vert &\leq
&C_{2}\left( \alpha ,\lambda ,1,\frac{a}{x}\right) \frac{\left( x-a\right)
^{\alpha +1}}{(ax)^{\alpha -1}x^{2}}\left[ \sup \left\{ \left\vert f^{\prime
}\left( x\right) \right\vert ,\left\vert f^{\prime }\left( a\right)
\right\vert \right\} \right] \\
&&+C_{3}\left( \alpha ,\lambda ,1,\frac{x}{b}\right) \frac{\left( b-x\right)
^{\alpha +1}}{(bx)^{\alpha -1}b^{2}}\left[ \sup \left\{ \left\vert f^{\prime
}\left( x\right) \right\vert ,\left\vert f^{\prime }\left( b\right)
\right\vert \right\} \right]
\end{eqnarray*}
\end{corollary}

\begin{corollary}
Under the assumptions of Theorem \ref{2.3} with $\alpha =1,$ the inequality (%
\ref{2-3}) reduced to the following inequality%
\begin{eqnarray*}
&&\left( \frac{ab}{b-a}\right) \left\vert I_{f,g}\left( x,\lambda ,\alpha
,a,b\right) \right\vert \\
&=&\left\vert \left( 1-\lambda \right) f(x)+\lambda \left[ \frac{%
b(x-a)f(a)+a(b-x)f(b)}{x(b-a)}\right] -\frac{ab}{b-a}\dint\limits_{a}^{b}%
\frac{f\left( u\right) }{u^{2}}du\right\vert \\
&\leq &\left( \frac{ab}{b-a}\right) \left\{ \frac{\left( x-a\right) ^{2}}{%
x^{2}}\left[ \sup \left\{ \left\vert f^{\prime }\left( x\right) \right\vert
^{q},\left\vert f^{\prime }\left( a\right) \right\vert ^{q}\right\} \right]
^{1/q}C_{2}\left( 1,\lambda ,1,\frac{a}{x}\right) \right. \\
&&\left. +\frac{\left( b-x\right) ^{2}}{b^{2}}\left[ \sup \left\{ \left\vert
f^{\prime }\left( x\right) \right\vert ^{q},\left\vert f^{\prime }\left(
a\right) \right\vert ^{q}\right\} \right] ^{1/q}C_{3}\left( 1,\lambda ,1,%
\frac{x}{b}\right) \right\} ,
\end{eqnarray*}%
specially for $x=H=2ab/(a+b)$, we get%
\begin{eqnarray*}
&&\left\vert \left( 1-\lambda \right) f\left( H\right) +\lambda \left( \frac{%
f(a)+f(b)}{2}\right) -\frac{ab}{b-a}\dint\limits_{a}^{b}\frac{f\left(
u\right) }{u^{2}}du\right\vert \leq \frac{b-a}{4ab} \\
&&\times \left\{ a^{2}C_{2}\left( 1,\lambda ,1,\frac{a+b}{2b}\right) \left(
\sup \left\{ \left\vert f^{\prime }\left( H\right) \right\vert
^{q},\left\vert f^{\prime }\left( a\right) \right\vert ^{q}\right\} \right)
^{\frac{1}{q}}\right. \\
&&+\left. H^{2}C_{3}\left( 1,\lambda ,1,\frac{2a}{a+b}\right) \left( \sup
\left\{ \left\vert f^{\prime }\left( H\right) \right\vert ^{q},\left\vert
f^{\prime }\left( b\right) \right\vert ^{q}\right\} \right) ^{\frac{1}{q}%
}\right\} .
\end{eqnarray*}
\end{corollary}

\begin{corollary}
In Theorem \ref{2.3},

(1) If we take $x=H=2ab/(a+b),\ \lambda =\frac{1}{3}$, then we get the
following Simpson type inequality for fractional integrals 
\begin{eqnarray*}
&&\left\vert \frac{1}{6}\left[ f(a)+4f\left( H\right) +f(b)\right] -\left( 
\frac{ab}{b-a}\right) ^{\alpha }2^{\alpha -1}\Gamma \left( \alpha +1\right) %
\left[ J_{1/H+}^{\alpha }\left( f\circ g\right) (1/a)+J_{1/H-}^{\alpha
}\left( f\circ g\right) (1/b)\right] \right\vert \\
&\leq &\frac{b-a}{4ab}\left\{ a^{2}\left[ \sup \left\{ \left\vert f^{\prime
}\left( H\right) \right\vert ^{q},\left\vert f^{\prime }\left( a\right)
\right\vert ^{q}\right\} \right] ^{1/q}C_{2}\left( \alpha ,\frac{1}{3},1,%
\frac{a}{x}\right) \right. \\
&&\left. +H^{2}\left[ \sup \left\{ \left\vert f^{\prime }\left( H\right)
\right\vert ^{q},\left\vert f^{\prime }\left( a\right) \right\vert
^{q}\right\} \right] ^{1/q}C_{3}\left( \alpha ,\frac{1}{3},1,\frac{x}{b}%
\right) \right\} ,
\end{eqnarray*}%
specially for $\alpha =1$, we get%
\begin{eqnarray*}
&&\left\vert \frac{1}{6}\left[ f(a)+4f\left( H\right) +f(b)\right] -\frac{ab%
}{b-a}\dint\limits_{a}^{b}\frac{f(u)}{u^{2}}du\right\vert \leq \frac{b-a}{4ab%
} \\
&&\times \left\{ a^{2}C_{2}\left( 1,\frac{1}{3},1,\frac{a+b}{2b}\right) %
\left[ \sup \left\{ \left\vert f^{\prime }\left( H\right) \right\vert
,\left\vert f^{\prime }\left( a\right) \right\vert \right\} \right] ^{\frac{1%
}{q}}\right. \\
&&\left. +H^{2}C_{3}\left( 1,\frac{1}{3},1,\frac{2a}{a+b}\right) \left[ \sup
\left\{ \left\vert f^{\prime }\left( H\right) \right\vert ^{q},\left\vert
f^{\prime }\left( b\right) \right\vert ^{q}\right\} \right] ^{\frac{1}{q}%
}\right\} .
\end{eqnarray*}%
(2) If we take $x=H=2ab/(a+b),\ \lambda =0,$ then we get the following
midpoint type inequality for fractional integrals 
\begin{eqnarray*}
&&\left\vert f\left( H\right) -\left( \frac{ab}{b-a}\right) ^{\alpha
}2^{\alpha -1}\Gamma \left( \alpha +1\right) \left[ J_{1/H+}^{\alpha }\left(
f\circ g\right) (1/a)+J_{1/H-}^{\alpha }\left( f\circ g\right) (1/b)\right]
\right\vert \\
&\leq &\frac{b-a}{4ab}\left\{ a^{2}C_{2}\left( \alpha ,0,1,\frac{a+b}{2b}%
\right) \left[ \sup \left\{ \left\vert f^{\prime }\left( H\right)
\right\vert ^{q},\left\vert f^{\prime }\left( a\right) \right\vert
^{q}\right\} \right] ^{\frac{1}{q}}\right. \\
&&\left. +H^{2}C_{3}\left( \alpha ,0,1,\frac{2a}{a+b}\right) \left[ \sup
\left\{ \left\vert f^{\prime }\left( H\right) \right\vert ^{q},\left\vert
f^{\prime }\left( b\right) \right\vert ^{q}\right\} \right] ^{\frac{1}{q}%
}\right\} ,
\end{eqnarray*}%
specially for $\alpha =1$, we get%
\begin{eqnarray*}
&&\left\vert f\left( H\right) -\frac{ab}{b-a}\dint\limits_{a}^{b}\frac{f(u)}{%
u^{2}}du\right\vert \\
&\leq &\frac{b-a}{4ab}\left( \frac{1}{2}\right) ^{1-\frac{1}{q}}\left\{ 
\frac{a^{2}H^{2}}{H^{2q}}C_{2}^{1/q}\left( 1,0,q,\frac{a+b}{2b}\right) \left[
\sup \left\{ \left\vert f^{\prime }\left( \frac{2ab}{a+b}\right) \right\vert
^{q},\left\vert f^{\prime }\left( a\right) \right\vert ^{q}\right\} \right]
^{\frac{1}{q}}\right. \\
&&\left. +\frac{b^{2}H^{2}}{b^{2q}}C_{3}^{1/q}\left( 1,0,q,\frac{2a}{a+b}%
\right) \left[ \sup \left\{ \left\vert f^{\prime }\left( \frac{2ab}{a+b}%
\right) \right\vert ^{q},\left\vert f^{\prime }\left( b\right) \right\vert
^{q}\right\} \right] ^{\frac{1}{q}}\right\} .
\end{eqnarray*}%
(3) If we take $x=H=2ab/(a+b)$, $\lambda =1,$ then we get the following
trapezoid type inequality for fractional integrals 
\begin{eqnarray*}
&&\left\vert \frac{f(a)+f(b)}{2}-\left( \frac{ab}{b-a}\right) ^{\alpha
}2^{\alpha -1}\Gamma \left( \alpha +1\right) \left[ J_{1/H+}^{\alpha }\left(
f\circ g\right) (1/a)+J_{1/H-}^{\alpha }\left( f\circ g\right) (1/b)\right]
\right\vert \\
&\leq &\frac{b-a}{4ab}\left\{ a^{2}C_{2}\left( \alpha ,1,1,\frac{a+b}{2b}%
\right) \left[ \sup \left\{ \left\vert f^{\prime }\left( H\right)
\right\vert ^{q},\left\vert f^{\prime }\left( a\right) \right\vert
^{q}\right\} \right] ^{\frac{1}{q}}\right. \\
&&\left. +H^{2}C_{3}\left( \alpha ,1,1,\frac{2a}{a+b}\right) \left[ \sup
\left\{ \left\vert f^{\prime }\left( H\right) \right\vert ^{q},\left\vert
f^{\prime }\left( b\right) \right\vert ^{q}\right\} \right] ^{\frac{1}{q}%
}\right\} ,
\end{eqnarray*}%
specially for $\alpha =1$, we get%
\begin{eqnarray*}
&&\left\vert \frac{f(a)+f(b)}{2}-\frac{ab}{b-a}\dint\limits_{a}^{b}\frac{f(u)%
}{u^{2}}du\right\vert \\
&\leq &\frac{b-a}{4ab}\left\{ a^{2}C_{2}^{1/q}\left( 1,1,1,\frac{a+b}{2b}%
\right) \left[ \sup \left\{ \left\vert f^{\prime }\left( H\right)
\right\vert ^{q},\left\vert f^{\prime }\left( a\right) \right\vert
^{q}\right\} \right] ^{\frac{1}{q}}\right. \\
&&\left. +H^{2}C_{3}^{1/q}\left( 1,1,1,\frac{2a}{a+b}\right) \left[ \sup
\left\{ \left\vert f^{\prime }\left( H\right) \right\vert ^{q},\left\vert
f^{\prime }\left( b\right) \right\vert ^{q}\right\} \right] ^{\frac{1}{q}%
}\right\} .
\end{eqnarray*}
\end{corollary}

\begin{corollary}
Let the assumptions of Theorem \ref{2.3} hold. If $\ \left\vert f^{\prime
}(x)\right\vert \leq M$ for all $x\in \left[ a,b\right] $ and $\lambda =0,$
then we get the following Ostrowski type inequality for fractional from the
inequality (\ref{2-3}) integrals%
\begin{eqnarray*}
&&\left\vert \left[ \left( \frac{x-a}{ax}\right) ^{\alpha }+\left( \frac{b-x%
}{bx}\right) ^{\alpha }\right] f(x)-\left( \frac{ab}{b-a}\right) ^{\alpha
}2^{\alpha -1}\Gamma \left( \alpha +1\right) \left[ J_{1/x+}^{\alpha }\left(
f\circ g\right) (1/a)+J_{1/x-}^{\alpha }\left( f\circ g\right) (1/b)\right]
\right\vert \\
&\leq &M\left[ \frac{\left( x-a\right) ^{\alpha +1}}{(ax)^{\alpha -1}x^{2}}%
C_{2}\left( \alpha ,0,1,\frac{a}{x}\right) +\frac{\left( b-x\right) ^{\alpha
+1}}{(bx)^{\alpha -1}b^{2}}C_{3}\left( \alpha ,0,1,\frac{x}{b}\right) \right]
.
\end{eqnarray*}
\end{corollary}

\begin{theorem}
\label{2.4}Let $f:$ $I\subseteq \left( 0,\infty \right) \rightarrow 
%TCIMACRO{\U{211d} }%
%BeginExpansion
\mathbb{R}
%EndExpansion
$ be a differentiable function on $I^{\circ }$ such that $f^{\prime }\in
L[a,b]$, where $a,b\in I^{\circ }$ with $a<b$. If $|f^{\prime }|^{q}$ is
harmonically quasi-convex on $[a,b]$ for some fixed $q>1$, then for $x\in
\lbrack a,b]$, $\lambda \in \left[ 0,1\right] $ and $\alpha >0$ the
following inequality holds%
\begin{eqnarray}
&&\left\vert I_{f,g}\left( x,\lambda ,\alpha ,a,b\right) \right\vert
\label{2-4} \\
&\leq &C_{1}^{1/q}\left( \alpha ,\lambda \right) \left\{ C_{2}^{1/p}\left(
\alpha ,\lambda ,p,\frac{a}{x}\right) \frac{\left( x-a\right) ^{\alpha +1}}{%
(ax)^{\alpha -1}x^{2p}}\left[ \sup \left\{ \left\vert f^{\prime }\left(
x\right) \right\vert ^{q},\left\vert f^{\prime }\left( a\right) \right\vert
^{q}\right\} \right] ^{1/q}\right.  \notag \\
&&\left. +C_{3}^{1/p}\left( \alpha ,\lambda ,p,\frac{x}{b}\right) \frac{%
\left( b-x\right) ^{\alpha +1}}{(bx)^{\alpha -1}b^{2p}}\left[ \sup \left\{
\left\vert f^{\prime }\left( x\right) \right\vert ^{q},\left\vert f^{\prime
}\left( b\right) \right\vert ^{q}\right\} \right] ^{1/q}\right\}  \notag
\end{eqnarray}%
where $C_{1},C_{2}$ and $C_{3}$ are defined as in Theorem \ref{2.2}.
\end{theorem}

\begin{proof}
Since$\left\vert f^{\prime }\right\vert ^{q}$ is harmonically quasi-convex
on $[a,b]$, from Lemma \ref{2.1}, using property of the modulus and the H%
\"{o}lder inequality we have%
\begin{eqnarray*}
&&\left\vert S_{f}\left( x,\lambda ,\alpha ,a,b\right) \right\vert \\
&\leq &\frac{\left( x-a\right) ^{\alpha +1}}{(ax)^{\alpha -1}}\left(
\dint\limits_{0}^{1}\frac{\left\vert t^{\alpha }-\lambda \right\vert }{%
A_{t}^{2p}(a,x)}dt\right) ^{\frac{1}{p}}\left(
\dint\limits_{0}^{1}\left\vert t^{\alpha }-\lambda \right\vert \left\vert
f^{\prime }\left( \frac{ax}{A_{t}(a,x)}\right) \right\vert ^{q}dt\right) ^{%
\frac{1}{q}} \\
&&+\frac{\left( b-x\right) ^{\alpha +1}}{(bx)^{\alpha -1}}\left(
\dint\limits_{0}^{1}\frac{\left\vert t^{\alpha }-\lambda \right\vert }{%
A_{t}^{2p}(b,x)}dt\right) ^{\frac{1}{p}}\left(
\dint\limits_{0}^{1}\left\vert t^{\alpha }-\lambda \right\vert \left\vert
f^{\prime }\left( \frac{bx}{A_{t}(b,x)}\right) \right\vert ^{q}dt\right) ^{%
\frac{1}{q}}
\end{eqnarray*}%
\begin{eqnarray*}
&\leq &C_{1}^{1/q}\left( \alpha ,\lambda \right) \left\{ C_{2}^{1/p}\left(
\alpha ,\lambda ,p,\frac{a}{x}\right) \frac{\left( x-a\right) ^{\alpha +1}}{%
(ax)^{\alpha -1}x^{2p}}\left[ \sup \left\{ \left\vert f^{\prime }\left(
x\right) \right\vert ^{q},\left\vert f^{\prime }\left( a\right) \right\vert
^{q}\right\} \right] ^{1/q}\right. \\
&&\left. +C_{3}^{1/p}\left( \alpha ,\lambda ,p,\frac{x}{b}\right) \frac{%
\left( b-x\right) ^{\alpha +1}}{(bx)^{\alpha -1}b^{2p}}\left[ \sup \left\{
\left\vert f^{\prime }\left( x\right) \right\vert ^{q},\left\vert f^{\prime
}\left( b\right) \right\vert ^{q}\right\} \right] ^{1/q}\right\}
\end{eqnarray*}%
which completes the proof.
\end{proof}

\begin{corollary}
Under the assumptions of Theorem \ref{2.4} with $\alpha =1,$ the inequality (%
\ref{2-4}) reduced to the following inequality%
\begin{eqnarray*}
&&\left( \frac{ab}{b-a}\right) \left\vert I_{f,g}\left( x,\lambda ,\alpha
,a,b\right) \right\vert \\
&=&\left\vert \left( 1-\lambda \right) f(x)+\lambda \left[ \frac{%
b(x-a)f(a)+a(b-x)f(b)}{x(b-a)}\right] -\frac{ab}{b-a}\dint\limits_{a}^{b}%
\frac{f\left( u\right) }{u^{2}}du\right\vert \\
&\leq &\left( \frac{ab}{b-a}\right) \left( \frac{2\lambda ^{2}-2\lambda +1}{2%
}\right) ^{\frac{1}{q}}\left\{ \frac{\left( x-a\right) ^{2}}{x^{2p}}\left[
\sup \left\{ \left\vert f^{\prime }\left( x\right) \right\vert
^{q},\left\vert f^{\prime }\left( a\right) \right\vert ^{q}\right\} \right]
^{1/q}C_{2}^{1/p}\left( 1,\lambda ,p,\frac{a}{x}\right) \right. \\
&&\left. +\frac{\left( b-x\right) ^{2}}{b^{2p}}\left[ \sup \left\{
\left\vert f^{\prime }\left( x\right) \right\vert ^{q},\left\vert f^{\prime
}\left( a\right) \right\vert ^{q}\right\} \right] ^{1/q}C_{3}^{1/p}\left(
1,\lambda ,p,\frac{x}{b}\right) \right\} ,
\end{eqnarray*}%
specially for $x=H=2ab/(a+b)$, we get%
\begin{eqnarray*}
&&\left\vert \left( 1-\lambda \right) f\left( H\right) +\lambda \left( \frac{%
f(a)+f(b)}{2}\right) -\frac{ab}{b-a}\dint\limits_{a}^{b}\frac{f\left(
u\right) }{u^{2}}du\right\vert \leq \frac{b-a}{4ab}\left( \frac{2\lambda
^{2}-2\lambda +1}{2}\right) ^{\frac{1}{q}} \\
&&\times \left\{ \frac{a^{2}H^{2}}{H^{2p}}C_{2}^{1/p}\left( 1,\lambda ,p,%
\frac{a+b}{2b}\right) \left( \sup \left\{ \left\vert f^{\prime }\left(
H\right) \right\vert ^{q},\left\vert f^{\prime }\left( a\right) \right\vert
^{q}\right\} \right) ^{\frac{1}{q}}\right. \\
&&+\left. \frac{b^{2}H^{2}}{b^{2p}}C_{3}^{1/p}\left( 1,\lambda ,p,\frac{2a}{%
a+b}\right) \left( \sup \left\{ \left\vert f^{\prime }\left( H\right)
\right\vert ^{q},\left\vert f^{\prime }\left( b\right) \right\vert
^{q}\right\} \right) ^{\frac{1}{q}}\right\} .
\end{eqnarray*}
\end{corollary}

\begin{corollary}
In Theorem \ref{2.4},

(1) If we take $x=H=2ab/(a+b),\ \lambda =\frac{1}{3}$, then we get the
following Simpson type inequality for fractional integrals 
\begin{eqnarray*}
&&\left( \frac{ab}{b-a}\right) ^{\alpha }\left\vert I_{f,g}\left( x,\lambda
,\alpha ,a,b\right) \right\vert \\
&\leq &\frac{b-a}{4ab}C_{1}^{1-1/q}\left( \alpha ,\frac{1}{3}\right) \left\{ 
\frac{a^{2}H^{2}}{H^{2p}}\left[ \sup \left\{ \left\vert f^{\prime }\left(
H\right) \right\vert ^{q},\left\vert f^{\prime }\left( a\right) \right\vert
^{q}\right\} \right] ^{1/q}C_{2}^{1/p}\left( \alpha ,\frac{1}{3},p,\frac{a+b%
}{2b}\right) \right. \\
&&\left. +\frac{b^{2}H^{2}}{b^{2p}}\left[ \sup \left\{ \left\vert f^{\prime
}\left( H\right) \right\vert ^{q},\left\vert f^{\prime }\left( a\right)
\right\vert ^{q}\right\} \right] ^{1/q}C_{3}^{1/p}\left( \alpha ,\frac{1}{3}%
,p,\frac{2a}{a+b}\right) \right\} ,
\end{eqnarray*}%
specially for $\alpha =1$, we get%
\begin{eqnarray*}
&&\left\vert \frac{1}{6}\left[ f(a)+4f\left( H\right) +f(b)\right] -\frac{ab%
}{b-a}\dint\limits_{a}^{b}\frac{f(u)}{u^{2}}du\right\vert \leq \frac{b-a}{4ab%
}\left( \frac{5}{18}\right) ^{\frac{1}{q}} \\
&&\times \left\{ \frac{a^{2}H^{2}}{H^{2p}}C_{2}^{1/p}\left( 1,\frac{1}{3},p,%
\frac{a+b}{2b}\right) \left[ \sup \left\{ \left\vert f^{\prime }\left(
H\right) \right\vert ,\left\vert f^{\prime }\left( a\right) \right\vert
\right\} \right] ^{\frac{1}{q}}\right. \\
&&\left. +\frac{b^{2}H^{2}}{b^{2p}}C_{3}^{1/p}\left( 1,\frac{1}{3},p,\frac{2a%
}{a+b}\right) \left[ \sup \left\{ \left\vert f^{\prime }\left( H\right)
\right\vert ^{q},\left\vert f^{\prime }\left( b\right) \right\vert
^{q}\right\} \right] ^{\frac{1}{q}}\right\} .
\end{eqnarray*}%
(2) If we take $x=H=2ab/(a+b),\ \lambda =0,$ then we get the following
midpoint type inequality for fractional integrals 
\begin{eqnarray*}
&&\left\vert f\left( H\right) -\left( \frac{ab}{b-a}\right) ^{\alpha
}2^{\alpha -1}\Gamma \left( \alpha +1\right) \left[ J_{1/H+}^{\alpha }\left(
f\circ g\right) (1/a)+J_{1/H-}^{\alpha }\left( f\circ g\right) (1/b)\right]
\right\vert \\
&\leq &\frac{b-a}{4ab}\left( \frac{1}{\alpha +1}\right) ^{\frac{1}{q}%
}\left\{ \frac{a^{2}H^{2}}{H^{2p}}C_{2}^{1/p}\left( \alpha ,0,p,\frac{a+b}{2b%
}\right) \left[ \sup \left\{ \left\vert f^{\prime }\left( H\right)
\right\vert ^{q},\left\vert f^{\prime }\left( a\right) \right\vert
^{q}\right\} \right] ^{\frac{1}{q}}\right. \\
&&\left. +\frac{b^{2}H^{2}}{b^{2p}}C_{3}^{1/p}\left( \alpha ,0,p,\frac{2a}{%
a+b}\right) \left[ \sup \left\{ \left\vert f^{\prime }\left( H\right)
\right\vert ^{q},\left\vert f^{\prime }\left( b\right) \right\vert
^{q}\right\} \right] ^{\frac{1}{q}}\right\} ,
\end{eqnarray*}%
specially for $\alpha =1$, we get%
\begin{eqnarray*}
&&\left\vert f\left( H\right) -\frac{ab}{b-a}\dint\limits_{a}^{b}\frac{f(u)}{%
u^{2}}du\right\vert \\
&\leq &\frac{b-a}{4ab}\left( \frac{1}{2}\right) ^{\frac{1}{q}}\left\{ \frac{%
a^{2}H^{2}}{H^{2p}}C_{2}^{1/p}\left( 1,0,p,\frac{a+b}{2b}\right) \left[ \sup
\left\{ \left\vert f^{\prime }\left( H\right) \right\vert ^{q},\left\vert
f^{\prime }\left( a\right) \right\vert ^{q}\right\} \right] ^{\frac{1}{q}%
}\right. \\
&&\left. +\frac{b^{2}H^{2}}{b^{2p}}C_{3}^{1/p}\left( 1,0,p,\frac{2a}{a+b}%
\right) \left[ \sup \left\{ \left\vert f^{\prime }\left( H\right)
\right\vert ^{q},\left\vert f^{\prime }\left( b\right) \right\vert
^{q}\right\} \right] ^{\frac{1}{q}}\right\} .
\end{eqnarray*}%
(3) If we take $x=H=2ab/(a+b)$, $\lambda =1,$ then we get the following
trapezoid type inequality for fractional integrals 
\begin{eqnarray*}
&&\left\vert \frac{f(a)+f(b)}{2}-\left( \frac{ab}{b-a}\right) ^{\alpha
}2^{\alpha -1}\Gamma \left( \alpha +1\right) \left[ J_{1/H+}^{\alpha }\left(
f\circ g\right) (1/a)+J_{1/H-}^{\alpha }\left( f\circ g\right) (1/b)\right]
\right\vert \\
&\leq &\frac{b-a}{4ab}\left( \frac{\alpha }{\alpha +1}\right) ^{\frac{1}{q}%
}\left\{ \frac{a^{2}H^{2}}{H^{2p}}C_{2}^{1/p}\left( \alpha ,1,p,\frac{a+b}{2b%
}\right) \left[ \sup \left\{ \left\vert f^{\prime }\left( H\right)
\right\vert ^{q},\left\vert f^{\prime }\left( a\right) \right\vert
^{q}\right\} \right] ^{\frac{1}{q}}\right. \\
&&\left. +\frac{b^{2}H^{2}}{b^{2q}}C_{3}^{1/p}\left( \alpha ,1,p,\frac{2a}{%
a+b}\right) \left[ \sup \left\{ \left\vert f^{\prime }\left( H\right)
\right\vert ^{q},\left\vert f^{\prime }\left( b\right) \right\vert
^{q}\right\} \right] ^{\frac{1}{q}}\right\} ,
\end{eqnarray*}%
specially for $\alpha =1$, we get%
\begin{eqnarray*}
&&\left\vert \frac{f(a)+f(b)}{2}-\frac{ab}{b-a}\dint\limits_{a}^{b}\frac{f(u)%
}{u^{2}}du\right\vert \\
&\leq &\frac{b-a}{4ab}\left( \frac{1}{2}\right) ^{\frac{1}{q}}\left\{ \frac{%
a^{2}H^{2}}{H^{2p}}C_{2}^{1/p}\left( 1,1,p,\frac{a+b}{2b}\right) \left[ \sup
\left\{ \left\vert f^{\prime }\left( H\right) \right\vert ^{q},\left\vert
f^{\prime }\left( a\right) \right\vert ^{q}\right\} \right] ^{\frac{1}{q}%
}\right. \\
&&\left. +\frac{b^{2}H^{2}}{b^{2p}}C_{3}^{1/p}\left( 1,1,p,\frac{2a}{a+b}%
\right) \left[ \sup \left\{ \left\vert f^{\prime }\left( H\right)
\right\vert ^{q},\left\vert f^{\prime }\left( b\right) \right\vert
^{q}\right\} \right] ^{\frac{1}{q}}\right\} .
\end{eqnarray*}
\end{corollary}

\begin{corollary}
Let the assumptions of Theorem \ref{2.4} hold. If $\ \left\vert f^{\prime
}(x)\right\vert \leq M$ for all $x\in \left[ a,b\right] $ and $\lambda =0,$
then we get the following Ostrowski type inequality for fractional from the
inequality (\ref{2-4}) integrals%
\begin{eqnarray*}
&&\left\vert \left[ \left( \frac{x-a}{ax}\right) ^{\alpha }+\left( \frac{b-x%
}{bx}\right) ^{\alpha }\right] f(x)-\left( \frac{ab}{b-a}\right) ^{\alpha
}2^{\alpha -1}\Gamma \left( \alpha +1\right) \left[ J_{1/x+}^{\alpha }\left(
f\circ g\right) (1/a)+J_{1/x-}^{\alpha }\left( f\circ g\right) (1/b)\right]
\right\vert \\
&\leq &\frac{M}{\left( \alpha +1\right) ^{\frac{1}{q}}}\left[ \frac{\left(
x-a\right) ^{\alpha +1}}{(ax)^{\alpha -1}x^{2p}}C_{2}^{1/p}\left( \alpha
,0,p,\frac{a}{x}\right) +\frac{\left( b-x\right) ^{\alpha +1}}{(bx)^{\alpha
-1}b^{2p}}C_{3}^{1/p}\left( \alpha ,0,p,\frac{x}{b}\right) \right] .
\end{eqnarray*}
\end{corollary}


\begin{thebibliography}{99}
\bibitem{ADDC10} M. Alomaria, M. Darus, S.S. Dragomir, P. Cerone, Ostrowski
type inequalities for functions whose derivatives are $s$-convex in the
second sense, Applied Mathematics Letters 23 (2010) 1071--1076.

\bibitem{ADK10} M.W. Alomari, M.Darus, U.S. Kirmaci, Refinements of
Hadamard-type inequalities for quasi-convex functions with applications to
trapezoidal formula and to special means, Computers and Mathematics with
Applications, 59 (2010) 225-232.

\bibitem{AH11} M. Alomari, S. Hussain, Two inequalities of Simpson type for
quasi-convex functions and applications, Applied Mathematics E-Notes,
11(2011), 110-117.

\bibitem{I07} D.A. Ion, Some estimates onthe Hermite-Hadamard inequality
through quasi-convex functions, Annals of University of Craiova, Math. Comp.
Sci. Ser. 34 (2007) 82-87.

\bibitem{I13} \.{I}. \.{I}\c{s}can, Hermite-Hadamard type inequalities for
harmonically convex functions, Hacet. J. Math. Stat., Accepted for
publication, 2013a.

\bibitem{I13a} \.{I}. \.{I}\c{s}can, On generalization of different type
inequalities for harmonically convex functions via fractional integrals,
preprint 2013.

\bibitem{I13b} \.{I}. \.{I}\c{s}can, Generalization of different type
integral inequalities for $s$-convex functions via fractional integrals,
Applicable Analysis. DOI:10.1080/00036811.2013.851785. Available online at:
http://dx.doi.org/10.1080/00036811.2013.851785.

\bibitem{I13c} \.{I}. \.{I}\c{s}can, New general integral inequalities for
quasi-geometrically convex functions via fractional integrals, Journal of
Inequalities and Applications, 2013, 2013:491,
doi:10.1186/1029-242X-2013-491.

\bibitem{I13d} \.{I}. \.{I}\c{s}can, On generalization of some integral
inequalities for quasi-convex functions and their applications,
International Journal of Engineering and Applied sciences (EAAS), 3 (1)
(2013), 37-42.

\bibitem{I13e} \.{I}. \.{I}\c{s}can, Generalization of different type
integral inequalities via fractional integrals for functions whose second
derivatives absolute values are quasi-convex, Konuralp journal of
Mathematics, Volume 1 No. 2 pp. 67--79 (2013)

\bibitem{KST06} A.A. Kilbas, H.M. Srivastava, J.J. Trujillo, Theory and
applications of fractional differential equations, Amsterdam, Elsevier, 2006.

\bibitem{SA11} M.Z. Sar\i kaya and N. Aktan, On the generalization of some
integral inequalities and \ their applications, Mathematical and Computer
Modelling, 54 (2011) 2175- 2182.

\bibitem{S12} E. Set, New inequalities of Ostrowski type for mapping whose
derivatives are $s$-convex in the second sense via fractional integrals,
Computers and Math. with Appl. 63 (2012) 1147-1154.

\bibitem{SO12} M.Z. Sar\i kaya and H. Ogunmez, On new inequalities via
Riemann-Liouville fractional integration, Abstract an Applied Analysis, vol.
2012, Article ID 428983, 10 pages, doi:10.1155/2012/428983.

\bibitem{SOS12} E. Set, M.E. Ozdemir and M.Z. Sar\i kaya, On new
inequalities of Simpson's type for quasi-convex functions with applications,
Tamkang Journal of Mathematics, 43 (3) (2012) 357-364.

\bibitem{SSYB11} M.Z. Sar\i kaya, E. Set, H. Yald\i z and N. Ba\c{s}ak,
Hermite-Hadamard's inequalities for fractional integrals and related
fractional inequalities, Math. Comput. Model. 57(9-10) (2013) 2403-2407

\bibitem{ZJF13} T.-Y. Zhang, , A.-P. Ji, and F. Qi, Integral inequalities of
Hermite-Hadamard type for harmonically quasi-convex functions, Proc.
Jangjeon Math. Soc. Vol. 16. No. 3. (2013) 399-407.
\end{thebibliography}
\end{document}